\documentclass[a4paper,12pt,englsh]{amsart}
\usepackage[top=3.0 cm,left=3.0 cm,right=3.0 cm]{geometry}
\usepackage{setspace}
\usepackage{amsmath,amsthm,amsfonts,amssymb,amscd,euscript,amstext,mathrsfs}
\usepackage{pxfonts}
\usepackage{psfrag}
\usepackage{srcltx}
\usepackage[colorlinks=false,linktocpage=true]{hyperref}
\usepackage[latin1]{inputenc}
\usepackage{epsfig}
\usepackage{lineno}
\usepackage{graphicx,color}

\usepackage[usenames,dvipsnames]{pstricks}
\usepackage{epsfig}
\usepackage{pst-grad} 
\usepackage{pst-plot} 

\usepackage{pstricks}
\usepackage{pst-func}
\usepackage{pst-math}
\usepackage{pst-plot}
\usepackage{pstricks-add}

\numberwithin{equation}{section}

\newtheorem{theorem}{Theorem}[section]
\newtheorem{corollary}{Corollary}[section]
\newtheorem{proposition}{Proposition}[section]
\newtheorem{lemma}{Lemma}[section]
\newtheorem{ask}{Question}[]
\newtheorem{definition}{Definition}[section]

\newtheorem{remark}{Remark}[section]

\newtheorem{open}{Problem}[]

\theoremstyle{plain}
\newtheorem{maintheorem}{Theorem}

\newtheorem{maincorollary}[maintheorem]{Corollary}

\author[]{}

\email{}

\urladdr{}

\author[]{}

\email{} 

\urladdr{}



\title{Phase Transitions for one-dimensional Lorenz-like expanding Maps}
\author{GOUVEIA, M. R. A. and OLER, J. G.}
\thanks{}
\date{\today}

\begin{document}
\maketitle

 
\medskip

\modulolinenumbers[1]

\begin{quote}{
\normalfont\fontsize{8}{10}\selectfont
{\bfseries Abstract.}  

Given  an one-dimensional Lorenz-like expanding map we prove that the condition\linebreak $P_{top}(\phi,\partial \mathcal{P},\ell)<P_{top}(\phi,\ell)$ (see, subsection \ref{Sub4} for definition), introduced by Buzzi and Sarig in (\cite{BuzSar2003}) is satisfied for all continuous potentials  $\phi:[0,1]\longrightarrow \mathbb{R}$.  We apply this to prove that quasi-H\"older-continuous potentials (see, subsection \ref{Sub2} for definition)  have at most one equilibrium measure and  we construct a family of continuous but not H\"older  and neither weak H\"older continuous potentials for which we observe phase transitions. Indeed, this class includes all H\"older and weak-H\"older continuous potentials and form an open and dense subset of $C([0,1],\mathbb{R})$, with the usual $\mathcal{C}^0$ topology. This give a certain generalization of the results proved  in \cite{PesZha2006}.  
}
\end{quote}

\tableofcontents

\newpage

\section{Introduction}\label{S1}

An one-dimensional Lorenz-like expanding map is a map $\ell:[0,1]\setminus\{d\}\rightarrow [0,1]$ ($d$ is the discontinuity point)  such that $\ell|_{(0,d)}$ and $\ell|_{(d,1)}$ are $C^{1+\alpha}$ diffeomorphisms with $|\ell^{'}|\geq \sqrt{2}$, and $\alpha>0$.   Since $\mathcal{P}=\{ (0,d),(d,1) \}$  is a partition of the interval $[0,1]$ we denote the map $\ell$ by $([0,1],\mathcal{P},\ell)$, in order to emphasize the importance of the partition $\mathcal{P}$. Given a continuous map $\phi:[0,1] \longrightarrow \mathbb{R}$  (also called potential function), a Borel invariant measure $\mu_{\phi}$ is called an equilibrium state for system $(\ell,\phi)$ if it is solution of the equation 

$$
\displaystyle P_{\rm top}(\ell, \phi)=\sup \left( \left\{ h_{\mu}(\ell)+\int_{[0,1]}\phi\, d\mu \, \left| \right. \, \mu \in M_{\rm inv}(\ell)\right\}  \right),
$$  
where $P_{\rm top}(\ell, \phi)$ denotes the topological pressure of $(\ell,\phi)$ (see \ref{p1}). 

In this context, Buzzi and Sarig (see \cite{BuzSar2003}) proved that if $\ell$ is an one-dimensional Lorenz-like expanding map and $\phi$ is a piecewise H\"older continuous potential such that 
\begin{equation}\label{star}
P_{top}(\phi,\partial\mathcal{P}, \ell)<P_{top}(\phi,\ell),
\end{equation}
then $(\ell,\phi)$ admits a unique equilibrium state.  The proof is based on representing the system as topologically transitive countable Markov shifts using their Markov diagrams and considering potentials with summable variations.

\begin{ask}
\rm If $\ell$ is one-dimensional Lorenz like expanding map, then  does every piecewise H\"older-continuous potential $\phi$ satisfy $P_{top}(\phi,\partial\mathcal{P}, \ell)<P_{top}(\phi,\ell)$?
\end{ask}
 
In \cite{BroOle2018} potentials no presenting property (\ref{star}) were not studied. However, the authors proved that if $\phi$ is a piecewise H\"older-continuous potential and  $\displaystyle\max\left\{ \limsup_{n \rightarrow \infty}\frac{1}{n}(S_{n}\phi)(0),\limsup_{n\rightarrow\infty}\frac{1}{n}(S_{n}\phi)(1)\right\}$ is an upper bound for topological pressure then $\phi$ admits a unique equilibrium measure.  Our first theorem gives a negative answer to Question 1. More precisely, we show that if $ \phi$ is continuous at $[0,1],$ then the system $(\ell, \phi)$ has the  property (\ref{star}).

\begin{maintheorem}\label{r1}
Let $([0,1],\mathcal{P}, \ell)$ be an one-dimensional Lorenz-like expanding map and consider $\phi:[0,1]\longrightarrow \mathbb{R}$ a continuous potential.  Then  $\phi$ satisfies \linebreak $P_{top}(\phi, \partial P, \ell)<P_{top}(\phi, \ell)$.   
\end{maintheorem}

See Section 2 for precise definitions. Using Theorem A above and the Theorem of [13], we obtain the following corollary:

\begin{maincorollary}\label{r2}
Let $([0,1],\mathcal{P}, \ell)$ be an one-dimensional Lorenz-like expanding map and $\phi:[0,1]\longrightarrow \mathbb{R}$ a potential continuous with summable variations.  Then $\phi$ admits a unique equilibrium measure. 
\end{maincorollary}

Considering $Th(A) $ as the set of potentials satisfying the hypothesis of Theorem \ref{r1} and $H^{\alpha}([0,1],\mathbb{R})$ being the set of H\"older-continuous potentials we have $ H^{\alpha}([0,1],R) \varsubsetneq Th(A) $.  At this point we can ask the question:

\begin{ask}
\rm Is there a class of potentials admitting a unique equilibrium measure which is larger then the class of piecewise continuous potentials?
\end{ask}

Pesin and Zhang in \cite{PesZha2006}  gave a positive answer to this question.  The authors showed that if $f$ is uniformly piecewise expanding map, then the class of potentials, admitting a unique equilibrium measure, is substantially larger then the class of H\"older continuous potentials.  More precisely, they describe a broad class formed by weakly H\"older continuous potentials, and this class is denote by $W^{\gamma}([0,1],\mathbb{R})$ which admit unique equilibrium measures and showed that this class includes some continuous but not H\"older continuous potentials. 

Here we show that the class of potentials admitting a unique equilibrium measure is substantially larger then the class $W^{\gamma}([0,1],\mathbb{R})$ studied in \cite{PesZha2006}.  

\begin{maincorollary}\label{r3}
If $\phi \in H^{\alpha}([0,1],\mathbb{R}) \cap W^{\gamma}([0,1],\mathbb{R})$, then $\phi \in Th(A)$.  In particular, $\phi$ admits a unique equilibrium measure.
\end{maincorollary} 

In \cite{PesZha2006}, based on Sarig's results, they construct a family of continuous (but not H\"older continuous) potentials $\varphi_{c}$ exhibiting phase transitions, i.e., there exists a critical value $c_{0} > 0$ such that for every $0 < c < c_0$ there is a unique equilibrium measure for $\varphi_{c}$ which is supported on $(0,1]$ and for $c > c_0$ the equilibrium measure is the Dirac measure at $0$.  Here we also build a family of potentials $ \phi_{t} \in Th(A)$ where the phase transition phenomenon occurs.  This is established in the following theorem.


\begin{maintheorem}\label{r4}
There exists a critical value $t_{0}>0$ such that the family of continuous potentials $\phi_{t}  \in Th(A)$ (note that  $\phi_{t} \notin \left(H^{\alpha}([0,1],\mathbb{R}) \cap W^{\gamma}([0,1],\mathbb{R})\right)$) presents phase transitions, i.e:
\begin{enumerate}
\item if $\displaystyle t \in \left(-\infty,t_{0} \right)$, then there are at least two equilibrium measures for $\phi_{t}$.  The Dirac measure at $p_{k}^{+}$ and $p_{k}^{-}$ are the equilibrium states;
\item if $ t = t_{0}$, then there is no equilibrium state for $\phi_{t_0}$;
\item if $ t \in (t_{0}, \infty) $, then there is a unique equilibrium measure for $\phi_{t}$.  This measure is supported on $(0,1)$. 
\end{enumerate}
\end{maintheorem}

We have already seen that $Th(A)$ contains a large amount of continuous potentials that have a single equilibrium measure, but they are not H\"older and are not weak-H\"older continuous. One of the main objectives of the theory of dynamic systems has been to study the typical properties of a system. Such studies can be directed toward understanding and discovering dynamic properties from both the topological and ergodic points of view, which are satisfied for a ``large" set of dense or residual systems.  In this context, Bronzi and Oler \cite{BroOle2018} proved that if $\ell$ is  an one-dimensional Lorenz-like expanding map and $H^{\alpha}([0,1],\mathbb{R})$ is the set of piecewise H\"older-continuous potentials of [0,1] with the usual $C^0$ topology, then  there exists an open and dense subset $\mathcal{H}$ of $H^{\alpha}([0,1],\mathbb{R})$ such that each $\phi \in \mathcal{H}$ admits exactly one equilibrium state.  The next theorem can be seen as a generalization of this result.

\begin{maintheorem}\label{r5} Let $\ell:[0,1]\setminus\{d\}\rightarrow [0,1]$ be an one-dimensional Lorenz-like expanding map and  $Th(A)$ the set of piecewise continuous potentials satisfying the hypotheses of Theorem A.  Then $Th(A)$ is an open and dense subset of $C([0,1],\mathbb{R})$, in the $C^r$  topology.
\end{maintheorem}

\begin{open}
Since $H \subset Th(A)$, $WH^{\gamma}([0,1],\mathbb{R}) \subset Th(A)$ and $\overline{Th(A)}=[0,1]$, is there a continuous potential $\phi$ admitting a unique equilibrium measure such that $\phi \notin Th(A)$?
\end{open}

\subsection*{Organization}
In Section \ref{S2} we review some standard facts about one-dimensional Lorenz-like expanding map and equilibrium states.  Section \ref{S3} is dedicated to the proof of Theorem \ref{r1}.  Corollary \ref{r2} is proved in Section \ref{S4}.  Corollaries \ref{r2} and \ref{r3} are proved in Sections \ref{S4} and \ref{S5} respectively. Sections \ref{S6} and \ref{S7} are devoted to proof  Theorems \ref{r4} and \ref{r5}.

\subsection*{Acknowledgment}
We thank Krerley Oliveira for proposing the problem and for many helpful suggestions during the preparation of the paper. Also, we thank Omri Sarig and Ali Messaoudi  for some helpful conversations and comments on the formulation and proof of Theorem \ref{r4}.

\section{Background and preliminary results}\label{S2}

\subsection{One-dimensional Lorenz-like expanding map}
Lorenz maps originally arise from the study of geometric models for the Lorenz equations \cite{G76, GW79, L63, S82, W79}.  This model induces an  \emph{one-dimensional Lorenz-like expanding map}.  Here we are conside\-ring the maps studied by Glendinning in \cite{Gl89}.

\begin{definition}\label{DefLor}
An one-dimensional Lorenz-like expanding map is a function \linebreak $\ell:[0,1] \rightarrow [0,1]$ satisfying the following
properties:
\begin{enumerate}
\item  $\ell$ has a unique discontinuity at $x=d$ and 
$\ell(d^{+})=\displaystyle\lim_{x \rightarrow d^{+}}\ell(x)= 0$, \linebreak $\ell(d^{-})=\displaystyle \lim_{x \rightarrow d^{-}}\ell(x) = 1$;
\item  For any $x \in [0, 1] \setminus \{d\}$, $\ell^{\prime}(x) > \sqrt{2}$ ;
\item Each inverse branch of  $\ell$ extends to a $C^{1+\theta}$, function on $[\ell(0),1]$ or on$[0,\ell(1)]$, for some $ \theta > 0$, and if $g$ denotes any of these inverse branches, then $g^{\prime}(x)=\lambda < 1.$ 
\end{enumerate}
\end{definition}

\begin{figure}[htb!]
\label{fig1}
\centering
\psfrag{A}{$0$}
\psfrag{B}{$d$}
\psfrag{C}{$1$}
\includegraphics[width=5cm,height=5cm]{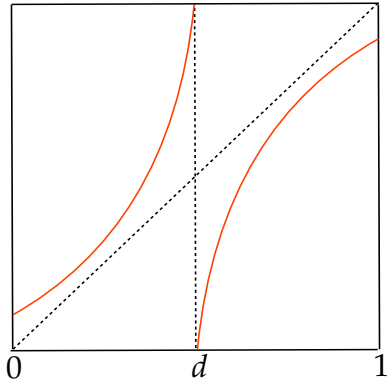}
\caption{One-dimensional Lorenz-like expanding map.}
\end{figure}

We denote by $\mathcal{P}$ {\it the natural partition} of $[0, 1]\backslash\{d\}$, \textit{i.e.}, \linebreak$\mathcal{P}=\{(0, d), (d, 1)\}.$  The {\it boundary} of $\mathcal{P}$ is $\partial \mathcal{P}:=\{0,d,1\}$.    Also, we define 
$$
\mathcal{P}^{(n)}=\{ C_{n} = P_0 \cap L ^{-1}(P_1) \cap \cdots \cap L^{-n+1}(P_{n-1}) \not= \emptyset \; | \;  P_{i} \in \mathcal{P} \}.
$$

In order to find a periodic point for an one-dimensional Lorenz-like expanding map  we define $C^{+}_{n}$ and $C^{-}_{n}$  as the cylinders on the right and left hand side of the discontinuity $d$, respectively, \textit{i.e.},  $d \in \partial C^{\pm}_{n}$.
For this purpose we introduce an auxiliary family $A_n$ by recursively as follows: let $A_0 := \left( d,1\right)$,

$$
\begin{array}{lll}
\text{ if } d\notin T^{i}(A_{0}) & \text{we definte}& A_i=T(A_{i-1}) \text{ for each } 0 < i \leq n-1
\end{array}
$$
 or
$$
\text{ if } d\in T^{i}(A_{0})\text{ we definte } \left\{
\begin{array}{ll}
 A_{n}=T^n(A_{0}),\\\\
A_{n+1}=T^n(A_n^{*}),
\end{array}
\right.
$$
where $A_n^*$ is the connected component of $A_n \setminus \{d\}$ which contains $T^n (d^+).$

\begin{definition}(see \cite{BroOle2018,GraGrz1998})\label{cut}
An integer $N$ is a \emph{cutting time} for $T$ if $d \in A_{N}$.
\end{definition}

We show that there are sequences of integers $\{N^{\pm}_{k}\}_{k \in \mathbb{N}}$ and sequences of periodic points $(p^{\pm}_{k})$ such that $p_{k}^{\pm} \xrightarrow{k\rightarrow \infty} d^{\pm}$, where $d^{\pm} \in \partial C_{N^{\pm}_{k}}$, and $\ell^{N^{\pm}_{k}}(p^{^{\pm}}_{k})=p^{^{\pm}}_{k}$ with $p^{^{\pm}}_{k} \in C_{N^{\pm}_{k}}$, $\forall\,k\geq 1$.
Moreover, we show that $P(\phi,\partial\mathcal{P},\ell)$ can be calculated by the average of $p^{\pm}_{k}$.

\begin{lemma}\label{cor1}(see \cite{BroOle2018})
For every $k \in \mathbb{N}$ there exist $p^{\pm}_{k} \in C_{N^{\pm}_{k}}$ such that $\ell^{N^{\pm}_{k}}(p^{\pm}_{k})=p^{\pm}_{k}$ .  
\end{lemma}

\subsection{Regularity of potential functions}\label{Sub2}

Let $\phi:[0,1] \rightarrow \mathbb{R}$ be a potential function.   A map $\phi$ is a {\it piecewise continuous potential} if its restriction to any element of $\mathcal{P}$ is  continuous. We denote 
$$
C_{\mathcal{P}}([0,1],\mathbb{R}):=\{\phi:[0,1] \rightarrow \mathbb{R}  \text{ such that } \phi \text{ is piecewise  continuous potential}\}.
$$

A map $\phi \in C([0,1],\mathbb{R}) $ is a {\it piecewise H\"older-continuous potential} if its restriction to any element of $\mathcal{P}$ is H\"older continuous, i.e., for all $x$, $y$ in the some element of $\mathcal{P}$, we have
$$|
\phi(x)-\phi(y)| \leq K |x-y|^{\alpha},
$$
for some constants $\alpha>0$ and $0<K<\infty$. 
Let us denote by $H^{\alpha}([0,1],\mathbb{R})$ the set
$$
H_{\mathcal{P}}^{\alpha}([0,1],\mathbb{R}):=\{\phi \in C([0,1],\mathbb{R})  \, : \, \phi \text{ is piecewise H\"older continuous potential}\}.
$$

We define the \textit{$n$-variation} of $\phi \in C([0,1],\mathbb{R})$ as
$$
\displaystyle V_{n}(\phi)=\sup_{x,y \, \in\, C_{n} \,\in \, \mathcal{P}^{n}}\left\{ \left| \phi(x) - \phi(y) \right| \right\}.
$$

We say that a potential $\phi \in C([0,1],\mathbb{R})$ has \textit{summable variations} if
$$
\displaystyle \sum_{n \geq 2}V_{n}(\phi) < +\infty.
$$

For simplicity of notation, we write
$$
SV([0,1],\mathbb{R}):=\{\phi \in C([0,1],\mathbb{R}) \text{ sucht that } \phi \text{ has summable variations}\}.
$$

We say that $\phi$ is \textit{weakly-H\"older-continuous} if there exist  $A>0$ and $0<\gamma < 1$ such that 
$$
\displaystyle V_{n}(\phi) \leq A \gamma^{n}, \text{ for all } n \geq 1. 
$$

With this notation, we have
$$
WH^{\gamma}([0,1],\mathbb{R}):=\{\phi:[0,1] \rightarrow \mathbb{R} \, | \, \phi \text{ is weak-H\"older continuous potential}\}.
$$

We say that $\phi$ is \textit{quasi-weakly-H\"older-continuous} if there exists  $0<\gamma < 1$ such that for all $n \geq 1$,
$$
\displaystyle V_{n}(\phi) \leq A(n) \gamma^{n},\text{ for all } n \geq 1,
$$
and $\displaystyle \lim_{n \rightarrow \infty}\frac{A(n+1)}{A(n)} < 1$.
Let us introduce the notation
$$
qWH^{\gamma}([0,1],\mathbb{R}):=\{\phi:[0,1] \rightarrow \mathbb{R} \text{ sucht that } \phi \text{ is quasi-weak-H\"older continuous potential}\}.
$$

\subsection{Birkhoff averages properties}
Given a map $\phi \in C([0,1],\mathbb{R})$, we have a notion of Birkhoff averages defined by
$$
\displaystyle (S_{n}\phi)(x)=\phi(x) + \phi(\ell(x))+ \cdots + \phi(\ell^{n-1}(x))=\sum_{j=1}^{n-1}\phi(\ell^{j}(x)).
$$  

We say that a potential $\phi:[0,1] \longrightarrow \mathbb{R}$ has bounded distortion if there exists a constant $C>0$ such that
$$
\left| \displaystyle \frac{1}{n}\left((S_{n}\phi)(x)-(S_{n}\phi)(y) \right)\right| \leq C, \text{ for all } x, y \in C_{n}\in \mathcal{P}^{(n-1)}.
$$

\begin{lemma}\label{distor}
Let $\ell:[0,1]\setminus\{d\}\rightarrow [0,1]$ be an one-dimensional Lorenz-like expanding map and $\phi :[0,1] \longrightarrow \mathbb{R}$ a continuous map. Then $\phi$ has bounded distortion.
\end{lemma}
\begin{proof}
\rm As $x, y \in C_{n} \in \mathcal{P}^{(n-1)}$ and ${\rm diam}(C_{n}) \longrightarrow 0$ when $n\rightarrow \infty$, for all $\delta$, there exists 
$$
n_{0}=n_{0}(\delta) \text{ such that }, \text{ then }\sup\left\{ d(\ell^{n}(x),\ell^{n}(y)) \right\}\leq \delta, \, \forall\, n \geq n_{0}.
$$

Since $\phi$ is a continuous function defined on $[0,1]$, we have that $\phi$ is uniformly continuous on $[0,1]$.  Thus by uniform continuity there exists $n_1 > 0$ such that, then 
$$
n \geq n_1 \Rightarrow \left| \phi(\ell^n(x)) - \phi(\ell^n(y))  \right| < \frac{\epsilon}{2}, \text{ for all } x, y \in C_{n}.
$$

Take $n_{2} > n_{1}$, such that $\displaystyle 2\frac{n_{1}M}{n_{2}} < \frac{\epsilon}{2}$, where 
$$
\displaystyle M=\max_{0 \leq j \leq n_{1}} \left\{ \left| \phi(\ell^{j}(x))  \right| , \left| \phi(\ell^{j}(y))  \right| \right\}.
$$

Therefore, if $n \geq n_2$, then
$$
\begin{array}{lll}
\left| \displaystyle \frac{1}{n}\left((S_{n}\phi)(x)-(S_{n}\phi)(y)\right) \right|  & = & \displaystyle \frac{1}{n}\left(\left|\sum_{j=0}^{n-1}\phi(\ell^{j}(x)) - \sum_{j=0}^{n-1}\phi(\ell^{j}(y))\right|\right) \\\\
&\leq & \displaystyle \frac{1}{n}\left( \sum_{j=0}^{n-1} \left|\phi(\ell^{j}(x)) - \phi(\ell^{j}(y))\right|\right)\\\\
& = &\displaystyle \frac{1}{n}\left(\displaystyle\sum_{j=0}^{n_{1}-1} \left|\phi(\ell^{j}(x)) - \phi(\ell^{j}(y))\right|+  \displaystyle\sum_{j=n_{1}}^{n-1} \left|\phi(\ell^{j}(x)) - \phi(\ell^{j}(y))\right|\right)\\\\
& \leq &\displaystyle \frac{1}{n} \left(\displaystyle\sum_{j=0}^{n_{1}-1} \left(\left|\phi(\ell^{j}(x)) \right| + \left|\phi(\ell^{j}(y)) \right|\right)\right)+ \displaystyle \frac{1}{n}\left( \displaystyle\sum_{j=n_{1}}^{n-1} \left|\phi(\ell^{j}(x)) - \phi(\ell^{j}(y))\right|\right)\\\\
& \leq & \displaystyle \frac{1}{n}\left(n_{1}M+n_{1}M\right) + \left(\frac{n-1-n_1}{n}\right) \cdot \frac{\epsilon}{2} < 2\frac{n_{1}M}{n}+ \frac{\epsilon}{2}\\\\
& < &\displaystyle 2\frac{n_{1}M}{n_{2}}+ \frac{\epsilon}{2} < \frac{\epsilon}{2} + \frac{\epsilon}{2} = \epsilon.
\end{array}
$$
\end{proof}

As $\ell$ is not continuous in  $d \in [0,1]$ we make the following convention: $S_n \phi(d^{+})$ is the right-hand side limit of the function $S_n \phi(z)$ at $d$ and $S_n \phi(d^{-})$ is the left-hand side limit of $S_n \phi(z)$ at $d$.  More precisely,  for any $n \in \mathbb{N}$ we define     
\begin{equation*}\label{lim0}
S_n \phi(d^{\pm}) = \displaystyle \lim_{z \rightarrow d^{\pm}} \sum_{i=0}^{n-1} \phi(\ell^i(z)).
\end{equation*}
By definition $\ell(d^+)=0$ and $\ell(d^{-}) =1$, so we conclude that:

$$
\displaystyle\limsup_{n\rightarrow \infty} \frac{1}{n} S_n \phi (d^{+}) =\limsup_{n \rightarrow \infty} \frac{1}{n} S_n \phi(0) \text{ and} \displaystyle\limsup_{n\rightarrow \infty} \frac{1}{n} S_n \phi (d^{-}) =\limsup_{n \rightarrow \infty} \frac{1}{n} S_n \phi(1).
$$

The following Lemma \ref{r91} guaranties that the above relations are well defined. 

\begin{lemma}\label{r91}(see \cite{BroOle2018})
Let $\ell:[0,1]\setminus\{d\}\rightarrow [0,1]$ be an one-dimensional Lorenz-like expanding map and consider $\phi \in  C^{\alpha}([0,1],\mathcal{P})$. 
\begin{itemize}
\item[$(i)$]  If does not exist  $n_{0}\in \mathbb{N}$ such that $\ell^{n_{0}}(0)=d$, then
\begin{equation*}\label{r01}
\displaystyle\limsup_{n \rightarrow \infty}\frac{1}{n}S_{n}\phi(d^{+})=\limsup_{n \rightarrow \infty}\frac{1}{n}S_{n}\phi(0).
\end{equation*}
\item[$(ii)$]If there exists $n_{0} \in \mathbb{N}$ such that $\ell^{n_{0}}(0)=d$, then
\begin{equation*}\label{r02}
\displaystyle\limsup_{n \rightarrow \infty}\frac{1}{n}S_{n}\phi(d^{+})=\displaystyle\frac{1}{n_{0}}S_{n_{0}}\phi(0).
\end{equation*}
\end{itemize}
The same conclusion holds for $d^-$ replacing $0$ for  $1$.
\end{lemma}

\begin{remark}
From now on  we use $\displaystyle \limsup_{n \rightarrow \infty}\frac{1}{n}S_{n}\phi(0)$ or $\displaystyle \limsup_{n \rightarrow \infty}\frac{1}{n}S_{n}\phi(1)$ to refer one of the items in the above Lemma \ref{r91}. 
\end{remark}

\begin{lemma}\label{per1}
Let $\ell:[0,1]\setminus\{d\}\rightarrow [0,1]$ be an one-dimensional Lorenz-like
expanding map and a potential $\phi \in C([0,1],\mathcal{P})$. If $N^{+}_{k} \in \mathbb{N}$ is such that $d \in \partial C_{N^{+}_{k}}$, $p^{+}_{k} \in C_{N^{+}_{k}}$  and $\ell^{N^{+}_{k}}(p^{+}_{k})=p^{+}_{k}$,
then 
\begin{equation*}\label{eq4}
\displaystyle \limsup_{k \rightarrow \infty}\frac{1}{N_{k}^{+}}S_{N_{k}^{+}}\phi(p_{k}^
{+}) = \limsup_{n \rightarrow \infty}\frac{1}{n}S_{n}\phi(0).
\end{equation*}

If $N^{-}_{k} \in \mathbb{N}$ is such that $d \in \partial C_{N^{-}_{k}}$,
$p^{-}_{k} \in C_{N^{-}_{k}}$  and $\ell^{N^{-}_{k}}(p^{-}_{k})=p^{-}_{k}$,
then 
\begin{equation*}\label{eq4}
\displaystyle \limsup_{k \rightarrow \infty}\frac{1}{N_{k}^{-}}S_{N_{k}^{-}}\phi(p_{k}^{-}) = \limsup_{n \rightarrow \infty}\frac{1}{n}S_{n}\phi(1).
\end{equation*}
\end{lemma}

\subsection{Topological Pressure}\label{Sub4}

According to Buzzi-Sarig \cite{BuzSar2003}, the {\it pressure of a subset} $S \subset [0,1]$ and a potential $\phi \in  qWH^{\gamma}([0,1],\mathcal{P})$ is defined by
\begin{equation}\label{p1}
	P_{top}(\phi, S, \ell)=\displaystyle \limsup_{n \rightarrow\infty}\frac{1}{n}\log\left( \sum_{\substack{C_{n} \in
			\mathcal{P}^{(n)}\, : \, S \cap \overline{C_{n}}\neq \emptyset}}\sup_{x \in C_{n}}e^{S_{n}\phi(x)}\right),
\end{equation} 
where  the Birkhoff average is well defined.  

The {\it topological pressure} of $L$ for  $\phi \in  C^{\alpha}([0,1],\mathcal{P})$ is defined by \linebreak $P(\phi,\ell)=P_{top}(\phi,[0,1],\ell)$.  

\begin{corollary}\label{prop.partial}
Let $L:[0,1]\setminus\{d\}\rightarrow [0,1]$ be an one-dimensional Lorenz-like expanding  map, $\phi \in  C_{\mathcal{P}}([0,1],\mathbb{R})$, and 
$$
M(\phi, L)=\displaystyle\max\left\{ \limsup_{n \rightarrow \infty}\frac{1}{n}(S_{n}\phi)(0),\limsup_{n \rightarrow \infty}\frac{1}{n}(S_{n}\phi)(1)\right\}.
$$
Then
$$
M(\phi, L)-C\leq P_{top}(\phi, \partial \mathcal{P}, L)\leq M(\phi, L)+C.
$$
\end{corollary}

\begin{proof}
We Just use Lemma \ref{distor} and make the superficial modifications to the proof of Proposition 3.1 of \cite{BroOle2018}.
\end{proof}

\begin{proposition}[see \cite{BuzSar2003}]\label{med}
Consider $\ell$ a piecewise expanding Lorenz-like map and a piecewise uniformly continuous potential $\phi$.  Let $\nu$ be an ergodic probability measure.  If $\nu(S)>0$, then 
$$
P_{\text{top}}(\phi,S,\ell)\geq h_{\nu}(\ell)+\int\; \phi\; d\nu,
$$ where $h_{\nu}(\ell)$ is the metric entropy of $\nu$.
\end{proposition}


Furthermore, let $\mathcal{M}_{\ell}([0,1])$ denote the collection of $\ell$-invariant Borel probability measures on $[0,1]$.  A measure $\mu_{\phi} \in \mathcal{M}_{\ell}([0,1])$  is called an \textit{equilibrium state} for $\phi$ if
$$ 
\sup_{\mu \in \mathcal{M}_{\ell}([0,1])}\left\{h_{\mu}(\ell)+\int \phi \, d\mu\right\}=h_{\mu_{\phi}}(\ell)+\int \phi \, d\mu_{\phi}.
$$

\begin{theorem}[\cite{BuzSar2003}]\label{sarig-buzzi}
Let $([0,1],\mathcal{P},\ell)$ be a piecewise expanding map such that for all non-empty open sets $U$ we have $\displaystyle \ell([0,1]) \subset  \displaystyle \bigcup_{k \geq 0}\ell^{k}(U)$, and consider a potential $\phi \in  C([0,1],\mathcal{P})$ satisfying $
P_{top}(\phi,\partial \mathcal{P},\ell)<P_{top}(\phi,\ell)$. Then $\displaystyle P_{top}(\phi,\ell)=\sup_{\mu \in \mathcal{M}_{\ell}([0,1])}\left\{h_{\mu}(\ell)+\int \phi \, d\mu\right\}$ and there is a unique measure equilibrium state for potential $\phi$.
\end{theorem}

\subsection{Countable Markov subshifits}

Here we follow the notations, definitions and results of \cite{BuzSar2003}.  Let $\rm dom(\ell^n) \subset [0,1]$ denote the domain of definition of $\ell^n$. The symbolic dynamics of $([0,1],\mathcal{P}, \ell)$ is the left-shift $\sigma$ defined on the set:
$$
\Sigma(\ell)=\overline{ \left\{C=(P_{0},P_{1},...) \in \mathcal{P}^{\mathbb{N}\cup\{0\}}\, : \, \exists\, x \in [0,1]\, , \forall\,  n \geq 0,  x \in {\rm  dom(\ell^n)} \text{ and } \ell^{n}(x) \in P_{n}\right\}},
$$
where  $\overline{\{\, \cdot \,   \}}$ denotes the closure in the compact space $\mathcal{P}^{\mathbb{N}\cup\{0\}}$.  For all $C \in \Sigma(\ell)$, the map $\pi :\Sigma(\ell) \longrightarrow [0,1]$ is defined by $\displaystyle \pi(C)=\bigcap_{n \geq 0}\overline{C_{n}}$. Indeed, we define $\Phi :  \Sigma(\ell) \longrightarrow  \mathbb{R}$ to be $\displaystyle \Phi(C)=\lim_{n \longrightarrow \infty }\inf\left(\phi\left(C_{n}\right)\right)$, where $\phi\, : \,[0,1] \longrightarrow \mathbb{R}$ is continuous potential. As $\ell$ is piecewise expanding, $\pi$ and $\Phi$ are well defined.  

\begin{proposition}(see \cite{BuzSar2003})\label{BS0}
Define $\Delta = (\pi^{-1}(\partial \mathcal{P}))$.  If $P_{top}(\phi,\partial \mathcal{P},\ell) < P_{top}(\phi,\ell)$ then 
$$
\displaystyle \pi\, : \, \Sigma(\ell)\backslash\left\{ \bigcup_{k \geq 0} \sigma^{k}\left( \Delta \right)\right\} \longrightarrow [0,1]\backslash \left\{ \bigcap_{k \geq 0} \ell^{-k}(\partial P)\right\}
$$
is a measure-theoric isomorphism and satisfies $\pi \circ \sigma = \ell \circ \pi$.  
\end{proposition}

A shift invariant probability measure $m$ is called an equilibrium measure for \linebreak $\Phi \, : \,\Sigma(\ell)  \longrightarrow \mathbb{R}$ if $h_{m}(\sigma) + \int \, \Phi\, dm$ is well defined and maximal.  The Gurevich pressure of $\Phi$ is given by
$$
P_{G}(\Phi,\sigma)=\displaystyle \lim_{n \longrightarrow \infty}\frac{1}{n}\log\left( \sum_{\sigma^{n}(x)=x}e^{(S_{n}\Phi_{n})(x)} 1_{[a]}(x)\right),
$$
where $a \in S$ is fixed, $[a]=\left\{ x \in \Sigma \, : \, x_{0}=a \right\}$ and $(S_{n}\Phi) = \displaystyle \sum_{i=0}^{n}\Phi \circ \sigma^{i}$.  

Let $P_{\sigma}(\Sigma)$ denote the collection of $\sigma-$invariant Borel probability measure on $\Sigma$.  The pressure of $m \in P_{\sigma}(\Sigma)$ is given by $P_{m}(\Phi,\sigma)=h_{m}(\sigma) + \int \, \Phi\, dm$.  Note that this is not always well-defined, $\Phi$ might not be integrable, or it might happen that $h_{m}(\sigma)=+\infty$ and $\int \, \Phi \, dm = -\infty$.  If $\sigma$ is topologically mixing and $\sup\left( \Phi \right)<\infty$, then
$$
P(\Phi)=\sup\left\{ P_{m}(\Phi) \, : \, m \in  P_{\sigma}(\Sigma), \, P_{m}(\Phi) \text{ is well defined }\right\}.
$$

The condition $\sup\left( \Phi \right) < \infty $ guarantees that $\int \, \Phi \, dm$ is well defined (possibly infinite), so the well defined condition reduces to a preclusion of the $m$ for which $h_{m}(\sigma )=\infty$ and $\int \, \Phi \, dm = -\infty$.  In this the authors prove the following:

\begin{theorem}(see \cite{BuzSar2003})
Let $(\Sigma,\sigma)$ be a topological transitive countable Markov shift and suppose $\Phi \, : \, \Sigma \longrightarrow \mathbb{R}$ satisfies $\sup \left( \Phi \right) <  \infty$, $P(\Phi,\sigma) < \infty$ and $\displaystyle \sum_{n \geq 0}{\rm var}_{n}(\Phi) < \infty$.  Then there exists at most one invariant probability measure $m$ such tha $\displaystyle \int \, \Phi\, dm$ is well defined and maximal.
 \end{theorem}

\section{Proof of Theorem \ref{r1}}\label{S3}

Let $\phi \in C_{\mathcal{P}}([0,1],\mathbb{R})$, be a piecewise potential.  The function $\phi_{k}^{\pm}$ is defined by $\phi_{k}^{\pm}(x)=\phi(x)+\phi_{0}^{\pm}(x)$, for all $x \in \mathcal{P}$, where the behavior of the family $\phi_{0}^{\pm}$ is defined as follows.  
Recall that by Corollary \ref{cor1}, there exists a subsequence $N^{\pm}_k \rightarrow \infty$ such that $d \in \partial C_{N^{\pm}_{k}}$, $p^{\pm}_{k} \in C_{N^{\pm}_{k}}$ and  $\ell^{N^{\pm}_{k}}(p_{k})=p^{\pm}_{k}$.  Let $I^{^{\pm}}_{j}=(\ell^{j}(p^{^{\pm}}_{k})-\delta^{\pm}_{k},\ell^{j}(p^{^{\pm}}_{k})+\delta^{\pm}_{k})$  be intervals,  where $0 \leq j\leq N^{\pm}_{k}-1$.  Since the orbit of $p^{\pm}_{k}$ is a finite set,  there exists $\delta^{\pm}_{k}>0$  such that $I^{^{\pm}}_{i} \cap I^{^{\pm}}_{j} = \emptyset$, for all $0 \leq j,i \leq N^{\pm}_{k}-1$ with $i \not= j$.  Fix $\delta^{\pm}_{k}>0$ and consider a periodic point $p^{\pm}_{k} \in [0,1]$ with period $N^{\pm}_{k}$, of the one-dimensional  Lorenz-like expanding map $\ell$.  

The function $\phi^{\pm}_{j,k}$ is defined by
$$
\phi^{\pm}_{j,k}(x)=\left\{ \begin{array}{ccc}
\exp\left(\displaystyle\frac{-1}{(x-\ell^{j}(p_{k}^{\pm})+\delta_{k}^{\pm})(x-\ell^{j}(p_{k}^{\pm})-\delta_{k}^{\pm})}\right)&,& x \in I^{\pm}_{j}\\\\
0                                                            &,& x \in (I^{\pm}_{j})^{c},
\end{array}\right.
$$    
for all $j \in \{ 0,1,\cdots, N_{k}^{\pm}-1 \}$ and $x \in [0,1]$.  Here for simplicity of notation, we write $M=\exp(\frac{1}{(\delta_{k}^{\pm})^{2}})$.  Note that for each $j \in \{ 0,1,\cdots, N_{k}^{\pm}-1 \}$, $\ell^{j}(p_{k}^{\pm})$ is a maximum local point of $\phi_{j,k}$ in $I_{j,k}$ and $\phi_{j,k}(\ell^{j}(p_{k}^{\pm}))=M$.  

Let $(a_{n})$ be a sequence of real numbers such that $a_{n}$  tends monotonically to zero and $\displaystyle \frac{1}{n}\sum_{j=0}^{\infty}a_{j}< \infty$  (e. g $a_{n}=\frac{1}{n}$).   In this way we construct $\phi^{\pm}_{j,k}$ by
$$
\phi^{\pm}_{j,k}(x)=\left\{ \begin{array}{ccc}
\frac{a_j}{M} \cdot \phi_{j}(x)&,& x \in I^{\pm}_{j}\\\\
0                                                            &,& x \in (I^{\pm}_{j})^{c}.
\end{array}\right.
$$    

Observe that the largest value assumed by $\phi^{\pm}_{j,k}$ on $I_{j}$ is $a_j$.  Finally we define
$$
\phi^{\pm}_{0}(x)=\left\{ \begin{array}{ccc}
\displaystyle \sum_{j=1}^{N_ {k}^{\pm}-1} \phi^{\pm}_{j,k}(x)&,& x \in \displaystyle\bigcup_{j=1}^{N_ {k}^{\pm}-1} I^{\pm}_{j}\\\\
0                                                            &,& x \in \left( \displaystyle \bigcup_{j=1}^{N_ {k}^{\pm}-1} I^{\pm}_{j}\right)^{c}.
\end{array}\right.
$$     


\begin{remark}\label{Obs0}
As the map $\phi^{\pm}_{0}$ is $C^{\infty}$ we have that $\phi^{^{\pm}}_{k}$ is uniformly continuous. Indeed,  as each $\phi_{j,k}$ is built on the orbit of the periodic point $p^{\pm}_{k}$, we have  $\displaystyle\phi^{\pm}_{0}(\ell^{j}(p^{\pm}_{k}))=\sum_{j=0}^{N^{\pm}_{k}-1}a_{j}$ and
$$
\displaystyle S_{N_{k}}\phi_{k}(p^{^{\pm}}_{k})=S_{N_{k}}\phi(p^{^{\pm}}_{k})+\phi_{0}(\ell^{j}(p^{\pm}_{k}))=S_{N_{k}}\phi(p^{^{\pm}}_{k})+\sum_{j=0}^{N^{\pm}_{k}-1}a_{j}.
$$
\end{remark}

By Proposition \ref{prop.partial} and Lemma \ref{per1}, we have 
\begin{equation}\label{des1}
P_{top}(\phi, S, \ell) = \displaystyle\max\left\{ \displaystyle \limsup_{k \rightarrow \infty}\frac{1}{N_{k}^{+}}S_{N_{k}^{+}}\phi(p_{k}^
{+}),\displaystyle \limsup_{k \rightarrow \infty}\frac{1}{N_{k}^{-}}S_{N_{k}^{-}}\phi(p_{k}^{-})\right\}.
\end{equation}
where $p^{+}_{k} \in C_{N^{+}_{k}}$, $p^{-}_{k} \in C_{N^{-}_{k}}$ are  such that $\ell^{N^{+}_{k}}(p^{+}_{k})=p^{+}_{k}$, $\ell^{N^{-}_{k}}(p^{-}_{k})=p^{-}_{k}$ and $ d \in \partial C_{N^{+}_{k}} \cap \partial C_{N^{-}_{k}}$.

First we consider $k$ fixed.   Considering $C_{N^{+}_{k}} \in \mathcal{P}^{(N^{+}_{k}-1)}$ as in Corollary \ref{cor1}, then there exists $p_{k}^{+} \in C_{N^+_{k}}$ such that $\ell^{N^{+}_{k}}(p^{+}_{k})=p_{k}^{+}$. Furthermore one can construct a measure $\mu_{k}^{+}(\cdot)=\left(\frac{1}{N^{+}_{k}}\sum_{j=0}^{N^{+}_{k}-1}\delta_{\ell^{j}(p_{k}^{+})}\right)(\cdot)$, where $\delta_{\ell^{j}(p_{k}^{+})}$ is the Dirac measure with $\delta_{\ell^{j}(p_{k}^{+})}(\ell^{j}(p_{k}^{+}))=1$, $j \in \{0,1,\cdots,N^{+}_{k}-1\}$.  As $\mu_{k}^{+}(C_{N_{k}})>0$ by Proposition \ref{med} we have
\begin{equation}\label{desA}
\begin{array}{lcl}
P_{top}(\phi_{k}^{+},\ell)
& \geq &P_{top}(\phi_{k}^{+}, C_{N^{+}_{k}},\ell) \geq h_{\mu_{k}^{+}}(\ell) +\int \phi_{k}^{+}\, d\mu_{k}^{+}\\\\
&=&\displaystyle \frac{1}{N^{+}_{k}}\left(\sum_{j=0}^{N^{+}_{k}-1}\phi_{k}^{+}(\ell^{j}(p^{+}_{k}))\right)=\frac{1}{N^{+}_{k}}\left(\sum_{j=0}^{N^{+}_{k}-1}(\phi+\phi^{\pm}_{0})(\ell^{j}(p^{+}_{k}))\right)\\\\
&=&\displaystyle\frac{1}{N^{+}_{k}}\left(\sum_{j=0}^{N^{+}_{k}-1}(\phi)(\ell^{j}(p^{+}_{k}))+\sum_{j=0}^{N^{+}_{k}-1}(\phi^{+}_{0})(\ell^{j}(p^{+}_{k}))\right)\\\\
&=&\displaystyle\frac{1}{N^{+}_{k}} S_{N_{k}}\phi(p^{^{\pm}}_{k})+\frac{1}{N^{+}_{k}}\left(\sum_{j=0}^{N^{+}_{k}-1}\phi^{-}_{0}(\ell^{j}(p^{+}_{k}))\right)\\\\
 &=&\displaystyle \frac{1}{N^{+}_{k}} S_{N_{k}}\phi(p^{^{\pm}}_{k})+\frac{1}{N^{+}_{k}}\sum_{j=0}^{N^{+}_{k}-1}a_{j} .
	\end{array}	
\end{equation}

Using the same arguments for $\displaystyle \mu_{k}^{-}(\cdot)=\left(\frac{1}{N^{-}_{k}}\sum_{j=0}^{N^{+}_{k}-1}\delta_{\ell^{j}(p_{k}^{-})}\right)(\cdot)$, where $\delta_{\ell^{j}(p_{k}^{-})}$ is the Dirac measure with $\delta_{\ell^{j}(p_{k}^{-})}(\ell^{j}(p_{k}^{-}))=1$, $j \in \{0,1,\cdots,N^{-}_{k}-1\}$. we obtain

\begin{equation}\label{desB}
\begin{array}{lcl}
P_{top}(\phi_{k}^{-},\ell)
& \geq & \displaystyle \frac{1}{N_{k}} S_{N_{k}}\phi(p^{-}_{k})+\frac{1}{N_{k}}\sum_{j=0}^{N^{\pm}_{k}-1}a_{j} .
	\end{array}	
\end{equation}

By using (\ref{desA}) and (\ref{desB}), we get 

\begin{equation}\label{des2}
	P_{top}(\phi_{k}^{\pm},\ell)\geq \max\left\{ \displaystyle\frac{1}{N^{+}_{k}}(S_{N^{+}_{k}}\phi)(p^{+}_{k})+\frac{1}{N^{+}_{k}}\sum_{j=0}^{N^{+}_{k}-1}a_{j} , \displaystyle\frac{1}{N^{-}_{k}}(S_{N^{-}_{k}}\phi)(p^{-}_{k})+\frac{1}{N^{-}_{k}}\sum_{j=0}^{N^{-}_{k}-1}a_{j}  \right\}.
\end{equation}

As 
$$
P_{top}(\phi,\ell)+P_{top}(\phi_{0}^{\pm},\ell) \geq P_{top}(\phi+\phi_{0}^{\pm},\ell) = P_{top}(\phi_{k}^{\pm},\ell).
$$
we have that

\begin{equation*}
P_{top}(\phi,\ell)+P_{top}(\phi_{0}^{\pm},\ell) \geq \max\left\{ \displaystyle\frac{1}{N^{+}_{k}}(S_{N^{+}_{k}}\phi)(p^{+}_{k})+\frac{1}{N^{+}_{k}}\sum_{j=0}^{N^{+}_{k}-1}a_{j} , \displaystyle\frac{1}{N^{-}_{k}}(S_{N^{-}_{k}}\phi)(p^{-}_{k})+\frac{1}{N^{-}_{k}}\sum_{j=0}^{N^{-}_{k}-1}a_{j}  \right\}.
\end{equation*}

Now we need to compute $P_{top}(\phi^{\pm}_{0},\ell)$. To do this, observe that
$$
\begin{array}{lll}
Z_{n}(\phi_{0}^{\pm},S)
&=&\displaystyle \sum_{\substack{C_{n} \in
			\mathcal{P}^{(n-1)}}}\sup_{x \in C_{n}}e^{S_{n}(\phi_{0}^{\pm})(x)}\\\\
&=&\displaystyle  \sum_{\substack{C_{n} \in
			\mathcal{P}^{(n-1)}\; : \; C_{n} \cap O(p^{\pm}) \not=\emptyset}}\sup_{x \in C_{n} }e^{S_{n}(\phi_{0}^{\pm})(x)}+\sum_{\substack{C_{n} \in
			\mathcal{P}^{(n-1)}\; : \; C_{n} \cap O(p^{\pm}) =\emptyset}}\sup_{x \in C_{n} }e^{S_{n}(\phi_{0}^{\pm})(x)}.
\end{array}
$$

By definition of $\phi_{0}^{\pm}$, we have $\phi_{0}^{\pm}(x)=0$, for all $x \notin O(p^{\pm})$.  Thus \linebreak 
$$
\displaystyle\sum_{\substack{C_{n} \in\mathcal{P}^{(n-1)}\; : \; C_{n} \cap O(p^{\pm}) =\emptyset}}\sup_{x \in C_{n} }e^{S_{n}(\phi_{0}^{\pm})(x)} = 0.
$$  

So
$$
Z_{n}(\phi_{0}^{\pm},S)
=\displaystyle  \sum_{\substack{C_{n} \in
			\mathcal{P}^{(n-1)}\; : \; C_{n} \cap O(p^{\pm}) \not=\emptyset}}\sup_{x \in C_{n} }e^{S_{n}(\phi_{0}^{\pm})(x)}.
$$

By Remark \ref{Obs0}, as  $\displaystyle\phi^{\pm}_{0}(\ell^{j}(p^{\pm}_{k}))=\sum_{j=0}^{N^{\pm}_{k}-1}a_{j}$, we obtain that

\begin{equation}\label{des4}
\begin{array}{lll}
\displaystyle \frac{1}{n}\log(Z_{n}(\phi_{0}^{\pm},S))&=&\displaystyle \frac{1}{n}\log\left( \displaystyle  \sum_{\substack{C_{n} \in
			\mathcal{P}^{(n-1)}\; : \; C_{n} \cap O(p^{\pm}) \not=\emptyset}}\sup_{x \in C_{n} }e^{S_{n}(\phi_{0}^{\pm})(x)}  \right)\\\\
			&\leq&\displaystyle \frac{1}{n}\log\left(N^{\pm}_{k}e^{n \sup_{k \in \mathbb{N}}\{a_{k}\}}\right)\\\\
			&=&\sup_{k \in \mathbb{N}}\{a_{k}\}.
\end{array}
\end{equation}

Letting $n \rightarrow \infty$ in the inequation  (\ref{des4}), for all $k\geq1$ we get
\begin{equation}\label{des3}
P_{top}(\phi^{\pm}_{0},\ell)\leq \sup_{k \in \mathbb{N}}\{a_{k}\}.
\end{equation}

Combining inequalities (\ref{des2}) and (\ref{des3}), we obtain
\begin{equation}\label{des5}
	P_{top}(\phi,\ell) + \sup_{k \in \mathbb{N}}\{a_{k}\} \geq \max\left\{ \displaystyle\frac{1}{N^{+}_{k}}(S_{N^{+}_{k}}\phi)(p^{+}_{k})+\frac{1}{N^{+}_{k}}\sum_{j=0}^{N^{+}_{k}-1}a_{j} , \displaystyle\frac{1}{N^{-}_{k}}(S_{N^{-}_{k}}\phi)(p^{-}_{k})+\frac{1}{N^{-}_{k}}\sum_{j=0}^{N^{-}_{k}-1}a_{j}  \right\}.
\end{equation}

Letting $k \rightarrow \infty$ at inequality (\ref{des5}), we have that

\begin{equation}\label{des6}
P_{top}(\phi,\ell)\geq \displaystyle\max\left\{ \displaystyle \limsup_{k \rightarrow \infty}\frac{1}{N_{k}^{+}}S_{N_{k}^{+}}\phi(p_{k}^
{+}),\displaystyle \limsup_{k \rightarrow \infty}\frac{1}{N_{k}^{-}}S_{N_{k}^{-}}\phi(p_{k}^{-})\right\}+\frac{3c}{2}.
\end{equation}

Combining (\ref{des1}) and (\ref{des6}) we obtain $P_{top}(\phi, S, \ell) < P_{top}(\phi,\ell)$.

\section{Proof of  Corollary \ref{r2}}\label{S4}

Let us first prove that $([0,1], \ell)$ is topologically transitive.   Consider $U, V \subset [0,1]$  non-empty open sets.  Being $\ell$ uniformly piecewise expanding map we could find $k_{0} \in \mathbb{N}$, such that $[0,1] \subset \ell^{n_{0}}(U) $.   As $V \subset [0,1] \subset \ell^{n_{0}}(U)$ then there exists $k_{0} \in \mathbb{N}$, such that $V \cap \ell^{n_{0}}(U) \not= \emptyset $. Thus we have proved that $\ell$ is topologically transitive.   In the second part of the proof we showed that $\sup\left( \Phi \right) < \infty$.  Note that,

$$
\begin{array}{lll}
|\Phi(C)|& = &\displaystyle \left| \lim_{n \longrightarrow \infty }\inf \left( \phi \left( C_{n} \right) \right) \right|
\leq  \displaystyle \lim_{n \longrightarrow \infty }\left| \inf\left(\phi(C_{n}) \right)\right|
\leq \displaystyle \lim_{n \longrightarrow \infty }\left| \sup\left(\phi(C_{n}) \right)\right|.
\end{array}
$$

As $\displaystyle \sup_{x \in \mathcal{P}}\left( \phi \right)<\infty$,  we have that $\left| \Phi(C) \right|< \infty$ and obtain $\displaystyle \sup_{C \in \Sigma(\ell)}\left(\Phi \right) < \infty$.

Now, we will prove that  $P_{G}(\Phi,\sigma)<\infty$.  To do this, observe that by Proposition \ref{BS0} we have that $P_{G}(\Phi,\sigma) = P_{top}(\phi, [0, 1], \ell)$.  Hence, we have
$$
\begin{array}{lll}
Z_{n}(\phi,[0,1])& = &\displaystyle \sum_{\substack{C_{n} \in
			\mathcal{P}^{(n-1)}\, : \, S \cap \overline{C_{n}}\neq \emptyset}}\sup_{x \in C_{n}}e^{S_{n}\phi(x)}\\\\
			& \leq & \displaystyle \# \left\{ C_{n} \in
			\mathcal{P}^{(n-1)}\, : \, S \cap \overline{C_{n}}\neq \emptyset\right\} e^{n \sup(\phi)}\\\\
			&\leq& 2^{n} e^{n \sup(\phi)}.
\end{array}
$$

Thus, 
$$
\begin{array}{lll}
P_{G}(\Phi,\sigma)&=&P_{top}(\phi, [0, 1], \ell)\\\\
&=&\displaystyle \limsup_{n \rightarrow\infty}\frac{1}{n}\log  Z_{n}(\phi,[0, 1])\\\\
& \leq&\displaystyle \limsup_{n \rightarrow\infty}\frac{1}{n}\log\left( 2^{n} e^{n \sup(\phi)} \right)\\\\
&=&\log(2)+\sup(\phi).
\end{array}
$$

As $\displaystyle \sup_{x \in [0,1]}\phi < \infty$ we obtain $P_{G}(\Phi,\sigma)<\log(2)+\sup(\phi)<\infty$.  The next step, is to prove that $\displaystyle \sum_{n \geq 0}V_{n}(\Phi) \leq \infty$.  If $C_{1}, C_{2}\in C_{n} \subset \Sigma(\ell)$, then:
$$
\begin{array}{lll}
\displaystyle \left| \Phi(C_{1})-\Phi(C_{2}) \right|\ & = &  \displaystyle \left| \lim_{n \longrightarrow \infty }\inf\left(\phi(C^1_{n}) \right) - \lim_{n \longrightarrow \infty }\inf\left(\phi(C^2_{n}) \right)\right|\\\\
& \leq & \displaystyle  \lim_{n \longrightarrow \infty }\left|\inf \left(\phi \left (C^1_{n}\right) \right) - \inf \left(\phi \left (C^2_{n}\right) \right)\right|.
\end{array}
$$
On the other hand, 
$$
\left|\inf \left(\phi \left (C^1_{n}\right) \right) -\inf \left(\phi \left (C^2_{n}\right) \right)\right| \leq |\phi(x_{1}) - \phi(x_{2})|,\, x_1 \in C^1_{n}, x_2 \in C^2_{n}.
$$

Thus,
$$
\begin{array}{lll}
\displaystyle \left| \Phi(C_{1})-\Phi(C_{2}) \right|\ 
& \leq & |\phi(x_{1}) - \phi(x_{2})|,\, x_1 \in C_{1}, x_2 \in C_ 2. 
\end{array}
$$

Therefore,
$$
\begin{array}{lll}
\displaystyle \sum_{n \geq 2}V_{n}(\Phi)& = &\displaystyle \sum_{n \geq 2} \displaystyle                   \sup_{C_{1}, C_{2}\in C_{n}} \left\{ \left| \Phi(C_{1})-\Phi(C_{2}) \right|\right\} \\\\
& \leq &\displaystyle \sum_{n \geq 2} \displaystyle \sup_{x_1 \in C_{1}, x_2 \in C_ 2} \left\{ \left| \phi(x_{1}) - \phi(x_{2})\right| \right\} \\\\
& \leq &\displaystyle \sum_{n \geq 2} \displaystyle \sup_{x_1, x_2 \in  C_{n} } \left\{ \left| \phi(x_{1}) - \phi(x_{2})\right|\right\} \\\\
& = &\displaystyle \sum_{n \geq 2}V_{n}(\phi).
\end{array}
$$

Since, $\displaystyle \sum_{n \geq 2}V_{n}(\phi) \leq \infty$, we get $\displaystyle \sum_{n \geq 2}V_{n}(\Phi) \leq \infty$.  As $(\Sigma,\sigma)$ satisfies the hypothesis of Proposition \ref{BS0} and $\sigma \circ \pi = \pi \circ \ell$, this finishes the proof of Corollary \ref{r2}. 

\section{Proof Corollary \ref{r3}}\label{S5}
Note that if $\phi \in H_\mathcal{P}^{\alpha}([0,1],\mathbb{R})$ then 
$$
\begin{array}{lll}
\displaystyle \sum_{n \geq 2}V_{n}(\Phi)& = &\displaystyle \sum_{n \geq 2} \displaystyle    \sup_{ x, y\in C_{n}}  \left\{ \left| \phi(x)-\phi(y) \right|\right\} \\\\
& \leq &\displaystyle \sum_{n \geq 2} \displaystyle \sup_{x, y\in C_{n}} \left\{\left| K\left| x-y \right|^{\alpha}\right|\right\}\\\\
& \leq &|K|\displaystyle \sum_{n \geq 2}  \lambda^{\alpha},\, (\text{see Definition \ref{DefLor}, item {3}})\\\\
& <&\infty.
\end{array}
$$
Now, if $\phi \in W^{\gamma}([0,1],\mathbb{R})$ then 
$\displaystyle \sum_{n \geq 2}V_{n}(\Phi)
 \leq A \displaystyle \sum_{n \geq 2}  \gamma^{\alpha}<\infty$.  Thus if \linebreak $\phi \in H_{\mathcal{P}}^{\alpha}([0,1],\mathbb{R}) \cap WH^{\gamma}([0,1],\mathbb{R})$ then $\phi \in SV([0,1],\mathbb{R})$.    

On the other hand, if $\phi \in H_\mathcal{P}^{\alpha}([0,1],\mathbb{R}) \cap WH^{\gamma}([0,1],\mathbb{R})$, then $\phi$ is continuous and applying the same arguments used in the proof of Theorem \ref{r1}, we have $P_{top}(\phi,\partial \mathcal{P},\ell)<P_{top}(\phi,\ell)$, which proves the Corollary \ref{r3}.

\section{Proof Theorem \ref{r4}}\label{S6}

For all $t \in \mathbb{R}$ define the one-parameter family of functions $\phi_{t}: [0,1] \rightarrow \mathbb{R}$, by 
$$
\phi_{t}(x) = \left\{
\begin{array}{ccc}
\displaystyle \frac{1}{[t(1-\log(x))]^{n}}&,&\displaystyle \quad x\in \left[\left(\bigcup_{n =0}^{\infty}\overline{C_{n}}\right) -\{ 0 \}\right]\\\\
                      0\hspace{1,8cm}&,& \quad x=0.
\end{array}
\right.
$$

We have divided the proof of the Theorem \ref{r4} into a sequence of Lemmas.

\begin{lemma}
For all $t \in \mathbb{R}$, we have $\phi_{t} \notin H^{\alpha}([0,1],\mathbb{R})$.
\end{lemma}
\begin{proof}

Suppose that there exists $s\in \mathbb{R}$ such that $\phi_{s}$ is a H\"older continuous map.  Thus $|\phi_{s}(x)-\phi_{s}(y)|\leq [\phi_{s}]_{\alpha}|x-y|^{\alpha}$, where $[\phi_{s}]_{\alpha}$ is a positive constant.  In particular, since $\phi_{s}(0)=0$, we have that $\displaystyle \frac{|\phi_{s}(x)|}{x^{\alpha}}\leq [\phi_{s}]_{\alpha}$.  Therefore, 

$$
\begin{array}{lll}
[\phi_{s}]_{\alpha} & \geq &\displaystyle \frac{|\phi_{s}(x)|}{x^{\alpha}}=\left|\frac{1}{x^{\alpha}[c(1-ln(x))]^n}\right|
\end{array}
$$
This is absurd, since the right-hand side diverges as $x \longrightarrow 0$.
\end{proof}

\begin{lemma}
For all $t \in \mathbb{R}$, we have $\phi_{t} \notin WH^{\gamma}([0,1],\mathbb{R})$.
\end{lemma}

\begin{proof}
Suppose that there exists $r\in \mathbb{R}$ such that $\phi_{r}$  is a weak H\"older continuous map.  Thus $Var_{n}(\phi_{r}) \leq A \gamma^{n}$, where $A$ is a positive constant.  Therefore, 
$$
\begin{array}{lll}
A & \geq &\displaystyle \frac{Var_{n}(\phi_{r})}{\gamma^{n}} =  \frac{\displaystyle\sup_{x,y \in C_{n}}\left| \phi_{r}(x)-\phi_{r}(y) \right|}{\gamma^{n}}. 
\end{array}
$$
In particular, since $\phi_{r}(0)=0$, we have that 

$$
\begin{array}{lll}
A & \geq & \displaystyle \frac{\displaystyle\sup_{x \in C_{n}}\left| \phi_{r}(x) \right|}{\gamma^{n}} \geq  \frac{  \left|\phi_{n}(x)  \right|}{\gamma^{n}} = \displaystyle \left|\frac{1}{\gamma^{n}[c(1-ln(x))]^n}\right|
\end{array}
$$
This is absurd, since the right-hand side diverges as $x \longrightarrow 0$.
\end{proof}

\begin{lemma}
For all $t\in \mathbb{R}$, $\phi_{t} \in qWH^{\gamma}([0,1],\mathbb{R})$.
\end{lemma}
\begin{proof}
Note that for all $t \in \mathbb{N}$,  $\phi_{t}$ is continuous in $[0,1]$ and so it is bounded.  Indeed, $\forall \, x \in (0,1]$,

$$
\begin{array}{lll}
\left(\dfrac{d \phi_{t}}{dx}\right)(x) & = & \dfrac{d}{dx}\left( \displaystyle \frac{1}{[t(1-\log(x))]^{n}} \right)\\\\
&=&\displaystyle \frac{n}{xt^{n}(1-\ln(x))^{n+1}}.
\end{array}
$$ 

Thus,


$$
\begin{array}{lll}
Var_{n}(\phi_{t}) 
&=&\displaystyle \sup_{x,y \in C_{n}}\left\{ \left| \phi_{t}(x) - \phi_{t}(y) \right| \right\}\\\\
&=&\displaystyle \sup_{x,y \in C_{n}}\left\{ \left| \left(\dfrac{d \phi_{t}}{dx}\right)(\overline{x}) (x - y) \right| \right\}\\\\
&=&\displaystyle \sup_{x,y \in C_{n}}\left\{ \left| \frac{n}{\overline{x}t^{n}(1-\ln(\overline{x}))^{n+1}} (x - y) \right| \right\}\\\\
&=&\displaystyle \left| \frac{n}{\overline{x}t^{n}(1-\ln(\overline{x}))^{n+1}} \right|\sup_{x,y \in C_{n}}\left\{ \left| x - y \right| \right\}\\\\
&=&\displaystyle \left| \frac{n}{\overline{x}t^{n}(1-\ln(\overline{x}))^{n+1}}  \right|diam(C_{n}),
\end{array}
$$
where $\bar{x}$ is a point between $x$ and $y$ given by the Mean Value Theorem.  By definition \ref{DefLor}, item {3} we get $diam(C_{n}) = \lambda^n$, where $\lambda \in (0,1)$.  Thus, 
$$
\begin{array}{lll}
Var_{n}(\phi_{t})&=&\displaystyle \left| \frac{n}{\overline{x}t^{n}(1-\ln(\overline{x}))^{n+1}} \right|\lambda^{n}\\\\
                     &=&A(n)\lambda^n,
\end{array}
$$
where $A(n)=\displaystyle \left| \frac{n}{\overline{x}t^{n}(1-\ln(\overline{x}))^{n+1}} \right|$.
 
As
$$
\begin{array}{lll}
\displaystyle \frac{A(n+1)}{A(n)} &=& \frac{\displaystyle \left| \frac{n+1}{\overline{x}t^{n+1}(1-\ln(\overline{x}))^{n+2}} \right|}{\displaystyle \left| \frac{n}{\overline{x}t^{n}(1-\ln(\overline{x}))^{n+1}} \right|} \\\\
&=&\left|\displaystyle \frac{(n+1)\overline{x}t^n(1-\ln(\overline{x}))^{n+1}}{n \overline{x}t^{n+1}(1-\ln(\overline{x}))^{n+2}}\right|\\\\
&=&\displaystyle\frac{n+1}{n} \frac{1}{t(1-\ln(\overline{x}))}
\end{array}
$$
we have $\displaystyle \lim_{n \rightarrow \infty}\frac{A(n+1)}{A(n)} <1$ and, therefore, $\phi_{t} \in qWH^{\gamma}([0,1],\mathbb{R})$.
\end{proof}

\begin{lemma}
If $\displaystyle t \in (t_{0}, \infty) $, where $\displaystyle t_{0} = \frac{\lambda}{1-ln(\overline{x})}$,  then $\phi_{t} \in 
SV(P,\mathbb{R})$.  So, there is a unique equilibrium measure for $\phi_{t}$.  This measure is supported on $(0,1)$. 
\end{lemma}

\begin{proof}

Note that

$$
\begin{array}{lll}
\displaystyle \sum_{n=2}^{\infty}Var_{n}(\phi_{t}) & =& \displaystyle \sum_{n=2}^{\infty} \left| \frac{n\lambda^{n} }{\overline{x}t^{n}(1-\ln(\overline{x}))^{n+1}} \right|\leq \infty
\end{array}
$$
because
$$
\begin{array}{lll}
\displaystyle\lim_{n \longrightarrow  \infty}\left(\frac{Var_{n+1}(\phi_{t})}{Var_{n}(\phi_{t})}\right)
&=&
\displaystyle\lim_{n \longrightarrow  \infty}\left|\frac{\displaystyle\frac{(n+1)\lambda^{n+1} }{\overline{x}t^{n+1}(1-\ln(\overline{x}))^{n+2}}}{\displaystyle \frac{n\lambda^{n} }{\overline{x}t^{n}(1-\ln(\overline{x}))^{n+1}}}\right|\\\\
&=&\displaystyle \lim_{n \longrightarrow  \infty}\left|\displaystyle \frac{(n+1)\lambda^{n+1}\overline{x}t^n(1-\ln(\overline{x}))^{n+1}}{n\lambda^n\overline{x}t^{n+1}(1-\ln(\overline{x}))^{n+2}}\right|\\\\
&=&\displaystyle \lim_{n \longrightarrow  \infty}\left( \frac{n+1}{n} \frac{\lambda}{t} \frac{1}{1-\ln(\overline{x})} \right)\\\\
&=&\displaystyle \frac{\lambda}{t} \frac{1}{1-\ln(\overline{x})}< 1.
\end{array}
$$

As $\phi_{t} \in SV(P,\mathbb{R})$ by Corollary \ref{r2}, we conclude that $\phi_{t}$ admits a unique equili\-brium measure.
\end{proof}

\begin{lemma}\label{l1}
If 
$\displaystyle t \in \left(-\infty,t_{0} \right)$, where $\displaystyle t_{0} = \frac{\lambda}{1-\ln(\overline{x})}$, then there are at least two equilibrium measures for $\phi_{t}$.  The Dirac measure at $p_{k}^{+}$ and $p_{k}^{-}$ are the equilibrium states.
\end{lemma}

\begin{proof}
Note that if 
$\displaystyle t \in \left(-\infty,t_{0} \right)$, then $\phi_{t} \notin
SV(P,\mathbb{R})$, i.e, 
$$
\begin{array}{lll}
\displaystyle \sum_{n=2}^{\infty}Var_{n}(\phi_{t}) & =& \displaystyle \sum_{n=2}^{\infty} \left| \frac{n\lambda^{n} }{\overline{x}t^{n}(1-\ln(\overline{x}))^{n+1}} \right|=\infty
\end{array}
$$
because
$$
\begin{array}{lll}
\displaystyle\lim_{n \longrightarrow  \infty}\left(\frac{Var_{n+1}(\phi_{t})}{Var_{n}(\phi_{t})}\right)
&=&\displaystyle \frac{\lambda}{t} \frac{1}{1-\ln(\overline{x})}> 1.
\end{array}
$$

On the other hand, we fixed $k$ and consider $p^{\pm}_{k} \in C_{N^{\pm}_{k}}$ such that $\ell^{N^{\pm}_{k}}(p^{\pm}_{k})=p^{\pm}_{k}$ (see Corollary \ref{cor1}).  Furthermore one can construct a measure \linebreak $\displaystyle \mu_{k}^{\pm}(\cdot)=\left(\frac{1}{N^{\pm}_{k}}\sum_{j=0}^{N^{\pm}_{k}-1}\delta_{\ell^{j}(p_{k}^{\pm})}\right)(\cdot)$, where $\delta_{\ell^{j}(p_{k}^{\pm})}$ is the Dirac measure with $\delta_{\ell^{j}(p_{k}^{\pm})}(\ell^{j}(p_{k}^{\pm}))=1$, $j \in \{0,1,\cdots,N^{\pm}_{k}-1\}$.  Indeed we can write   $n=q^{+}_{n} N_{k}^{+}+r^{+}_{n}, \; 0 \leq r^{+}_{n} \leq N_{k}^{+}-1.$  Hence, 

$$
\begin{array}{lll}
P_{top}(\phi_{t},\ell)&=&P_{top}(\phi_{t},[0,1],\ell) \geq P_{top}(\phi_{t},C_{N^{\pm}_{k}},\ell)\\\\
&=&\displaystyle \limsup_{n \rightarrow\infty}\frac{1}{n}\log\left( \sum_{\substack{C_{n} \in
			\mathcal{P}^{(n)}\, : \, C_{N^{\pm}_{k}} \cap \overline{C_{n}}\neq \emptyset}}\sup_{x \in C_{n}}e^{S_{n}\phi_{t}(x)}\right)\\\\
&=&\displaystyle \limsup_{n \rightarrow\infty}\frac{1}{n}\log\left( \sup_{x \in C^{\pm}_{n_{k}}}e^{S_{n}\phi_{t}(x)}\right)\\\\
&\geq&\displaystyle \limsup_{n \rightarrow\infty}\frac{1}{n}\log\left( e^{S_{n}\phi_{t}(p_{k}^{\pm})}\right)\\\\
&=&\displaystyle \limsup_{n \rightarrow\infty}\frac{1}{n}S_{n}\phi_{t}(p_{k}^{\pm})\\\\
&=&\displaystyle \lim_{n \rightarrow
	\infty}\left(\frac{q^{\pm}_{n}}{q^{\pm}_{n} N_{k}^{\pm}+r^{\pm}_{n}}
S_{N_{k}^{+}}\phi_{t}(p_{k}^{\pm})+\frac{1}{q^{\pm}_{n} N_{k}^{\pm}+r^{\pm}_{n}}S_{r^{\pm}_{n}}\phi_{t}(p_{k}^{\pm})\right)\\\\
&=&\displaystyle \frac{1}{N_{k}^{\pm}}S_{N_{k}^{\pm}}\phi_{t}(p_{k}^{\pm})=\displaystyle\sum_{j=0}^{N_{k}^{\pm}-1}\frac{1}{N_{k}^{\pm}[t(1-\ln(\ell^{j}(p_{k}^{\pm})))]^{N_{k}^{\pm}}}=\infty.
\end{array}
$$

Thus, 
$$
\begin{array}{lll}
h_{\mu^{\pm}}(\ell)+\int\; \phi_{t}\; d\mu^{\pm}&=&\displaystyle \frac{1}{N^{\pm}_{k}}\left(\sum_{j=0}^{N^{+}_{k}-1}\phi_{t}(\ell^{j}(p^{\pm}_{k}))\right)\\\\
&=&\displaystyle \frac{1}{N^{+}_{k}}\sum_{j=0}^{n-1}\frac{1}{[t(1-ln(\ell^{j}(p_{k}^{\pm})))]^{n}}=\infty
\end{array}
 $$

So, 
$$
P_{top}(\phi_{t},\ell) = \infty = h_{\mu^{\pm}}(\ell)+\int\; \phi_{t}\; d\mu^{\pm}.
$$
\end{proof}

\begin{lemma}
If $\displaystyle t = \frac {\lambda} {1-\ln(\overline{x})} $, then there is no equilibrium state.
\end{lemma}

\begin{proof}
If $\displaystyle t = \frac {\lambda} {1-\ln(\overline{x})} $, then $\displaystyle \phi(x)=\left( \frac{1-\ln(\overline{x})}{\lambda} \right)^{n}\left(\frac{1}{1-\ln(x)}\right)^{n}$.  Indeed,
$$
\begin{array}{lll}
\displaystyle \sum_{n=2}^{\infty}Var_{n}(\phi_{t}) & =& \displaystyle \sum_{n=2}^{\infty} \left| \frac{n}{\overline{x}(1-\ln(\overline{x}))} \right|=\infty.
\end{array}
$$

Suppose that
$$ 
P_{top}(\phi,\ell)=\sup_{\mu \in \mathcal{M}_{\ell}(X)}\left\{h_{\mu}(\ell)+\int \phi \, d\mu\right\}.
$$

As $\phi$ is continuous, we get $\sup(\phi)<\infty$.  Thus there exists $M>0$ such that $P_{top}(\phi,\ell)\leq M$.  Repeating the same arguments used in the proof of Lemma \ref{l1}, we obtain
$$
\begin{array}{lll}
P_{top}(\phi,\ell)& \geq& h_{\mu^{\pm}}(\ell)+\int\; \phi\; d\mu^{\pm}\\\\
&=&\displaystyle \frac{1}{N^{\pm}_{k}}\left(\sum_{j=0}^{N^{\pm}_{k}-1}\phi_{t}(\ell^{j}(p^{\pm}_{k}))\right)\\\\
&=&\displaystyle \frac{1}{N^{+}_{k}}\sum_{j=0}^{N^{\pm}_{k}-1}\left( \frac{1-\ln(\overline{x})}{\lambda} \right)^{N^{\pm}_{k}}\left(\frac{1}{1-\ln(\ell^{j}(p_{k}^{\pm}))}\right)^{N^{\pm}_{k}}\\\\
&\geq&\displaystyle \frac{1}{N^{\pm}_{k}}\sum_{j=0}^{N^{\pm}_{k}-1}\left(\frac{1}{1-\ln(\ell^{j}(p_{k}^{\pm}))}\right)^{N^{\pm}_{k}}
\end{array}
 $$

So, doing $k \longrightarrow \infty$ we have $M \geq P_{top}(\phi_{t},\ell) \geq  \infty$ which is a contradiction.   Therefore, the variational principle cannot occurs. This shows that the potential $\phi$ has no equilibrium state.
\end{proof}

\section{Proof Theorem \ref{r5}}\label{S7}
Consider $\phi \in Th(A)$.  By definition of $Th(A)$ we have that
$$
\begin{array}{lll}
Th(A) &=& \left\{ \phi \in C_{\mathcal{P}}([0,1], \mathbb{R}) \, : \, \phi \in SV([0,1],\mathbb{R}) \text{  and } P_{top}(\phi, \partial P, \ell)<P_{top}(\phi,\ell) \right\}\\\\
        &=&S(\phi) \cap D(\phi),
\end{array}
$$ 
where
$$
\begin{array}{lll}
S(\phi) &=&\left\{ \phi \in C_{\mathcal{P}}([0,1], \mathbb{R}) \, \left| \, \displaystyle \sum_{n \geq 2}V_{n}(\overline{\phi})<\infty \right\}\right.,
\end{array}
$$
and
$$
\begin{array}{lll}
D(\phi)&=&\left\{ \phi \in C_{\mathcal{P}}([0,1], \mathbb{R}) \, \left| \,  P_{top}(\phi, \partial P, \ell)<P_{top}(\phi,\ell) \right\}\right..
\end{array}
$$

In order to prove that $Th(A)$ is open and dense in $C_\mathcal{P}([0,1],\mathbb{R})$  it is sufficient to show that $S(\phi)$ and $D(\phi)$  are open and dense in $C_\mathcal{P}([0,1],\mathbb{R})$.  By Theorem A in \cite{BroOle2018}, $D(\phi)$ is open and dense in $C_\mathcal{P}([0,1],\mathbb{R})$.  Thus we only need to show that $S(\phi)$ is open and dense in $C_{\mathcal{P}}([0,1],\mathbb{R})$.  For this we observe that
$$
\begin{array}{lll}
S(\phi) &=&\left\{ \phi \in C_{\mathcal{P}}([0,1], \mathbb{R}) \, \left| \, \displaystyle \sum_{n \geq 2}V_{n}(\overline{\phi})<\infty \right\}\right.\\\\
&=&\left\{ \phi \in C_{\mathcal{P}}([0,1], \mathbb{R}) \, \left| \, \displaystyle \lim_{n\longrightarrow \infty}\sum_{j= 2}^{n}V_{j}(\phi)=s \right\}\right.\\\\
&\subset&\left\{ \phi \in C_{\mathcal{P}}([0,1], \mathbb{R}) \, \left| \, \forall \, n\geq 0, \exists \, M>0\, : \, \displaystyle \left|\sum_{j= 2}^{n}V_{j}(\phi)\right|<M\right\}\right..
\end{array}
$$

Fix $n\geq 0$ and define $F,\,G\,:\, C_{\mathcal{P}}([0,1],\mathbb{R}) \longrightarrow  \mathbb{R}$, by $F(\phi) =V_{n}(\phi)$ and \linebreak $\displaystyle G(\phi)=M-\left| \sum_{j=2}^{n}V_{j}(\phi) \right|$.  If $\phi_{0},\phi_{1}, \phi_{2} \in C(\mathcal{P},\mathbb{R})$ then
$$
\begin{array}{lll}
\left\| F(\phi_{1}) - F(\phi_{2}) \right\|_{C^{0}}&=&\left\| V_{n}(\phi_{1}) - V_{n}(\phi_{2}) \right\|_{C^{0}}\\\\
&=&\displaystyle \left\|\sup_{ x,y \in C_{n} }\left\{ \left| \phi(x) - \phi(y) \right|\right\}-\sup_{ x,y \in C_{n} }\left\{ \left| \phi(x) - \phi(y) \right|  \right\}\right\|_{C^{0}}\\\\
&=&\left\| \,\left\| (\phi_{1})_{C_{n}} \right\|-\left\| (\phi_{2})_{C_{n}} \right\|\, \right\|_{C^{0}}\leq \left\| (\phi_{1} -\phi_{2})_{C_{n}} \right\|_{C^{0}}\\\\
&\leq&\left\| \phi_{1} -\phi_{2}\right\|_{C^{0}}
\end{array}
$$
and
$$
\begin{array}{lll}
\displaystyle \lim_{\phi\rightarrow \phi_{0}}G(\phi)& = & \displaystyle \lim_{\phi\rightarrow \phi_{0}}\left( M-\left| \sum_{j=2}^{n}V_{j}(\phi) \right|\right)= M-\lim_{\phi\rightarrow \phi_{0}}\left(\left| \sum_{j=2}^{n}V_{j}(\phi) \right|\right)\\\\
&=&\displaystyle M-\left| \sum_{j=2}^{n}\left(\lim_{\phi\rightarrow \phi_{0}}V_{j}(\phi)\right) \right|=M-\left| \sum_{j=2}^{n} V_{j}(\phi_{0}) \right|\\\\
&=&G(\phi_{0}).
\end{array}
$$

Thus 
$$
S(\phi) \subset \left\{ \phi \in C_{\mathcal{P}}([0,1], \mathbb{R}) \, \left| \, \forall \, n\geq 0, \exists \, M>0\, : \, \displaystyle \left|\sum_{j= 2}^{n}V_{j}(\phi)\right|<M\right\}\right.=G^{-1}(\, (0,+\infty)  \,).
$$

Recall that by Corollary \ref{per1}, there exists a subsequence $N^{+}_k \rightarrow \infty$ such that $d \in \partial C_{N^{+}_{k}}$, $p^{\pm}_{k} \in C_{N^{\pm}_{k}}$ and  $L^{N^{\pm}_{k}}(p_{k})=p^{\pm}_{k}$.  Let $I^{^{\pm}}_{j}=(L^{j}(p^{^{\pm}}_{k})-\delta^{\pm}_{k},L^{j}(p^{^{\pm}}_{k})+\delta^{\pm}_{k})$  be intervals,  where $0 \leq j\leq N^{\pm}_{k}-1$.  Since the orbit of $p^{\pm}_{k}$ is a finite set,  there exists $\delta^{\pm}_{k}>0$  such that $I^{^{\pm}}_{i} \cap I^{^{\pm}}_{j} = \emptyset$, for all $0 \leq j,i \leq N^{\pm}_{k}-1$ with $i \not= j$.  Consider
$$
	B^{^{\pm}}_{\epsilon,k}(x)=\left\{ \begin{array}{ccc}
	\displaystyle \sum_{j=0}^{N^{\pm}_{k}-1} B^{^{\pm}}_{\epsilon,k,j}(x)&,& x \in \displaystyle\bigcup_{j=0}^{N^{\pm}_ k-1} I^{^{\pm}}_{j}\\\\
	0 & , & otherwise.
	\end{array}\right.
	$$  
where $B^{^{\pm}}_{\epsilon, k,j}$ is a bump function defined by figure \ref{Bump}.

\begin{figure}[htb!]
\caption{Bump function $B^{^{\pm}}_{\epsilon, k,j}$.}
\label{Bump}
\centering
\psfrag{A}{$L^{j}(p_{k})$}
\psfrag{B}{$L^{j}(p_{k})+\delta_{k}$}
\psfrag{C}{$L^{j}(p_{k})-\delta_{k}$}
\psfrag{D}{$\epsilon$}
\psfrag{E}{$L^{j}(0)$}
\psfrag{F}{$\delta_{k}$}	
\includegraphics[width=10cm,height=5cm]{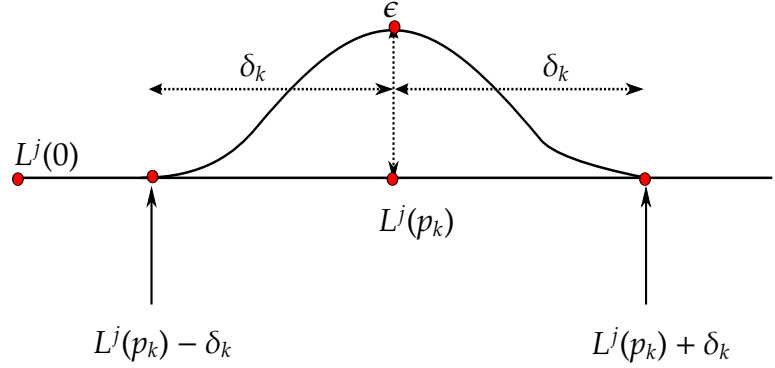}
\end{figure}

Consider $\phi^{^{\pm}}_{\epsilon,k}(x)=\phi(x)+B^{^{\pm}}_{\epsilon,k}(x)$, where $\phi \in S(\phi)$.  Then we have the following properties:
	\begin{enumerate}
		\item $\phi^{^{\pm}}_{\epsilon,k}$ is continuous;
		\item $\|\phi^{^{\pm}}_{\epsilon,k}-\phi\|_{C^{0}}<\epsilon$, for  all $\phi \in  C^{\alpha}([0,1],\mathcal{P})$.
	\end{enumerate}

We need show that $\phi^{^{\pm}}_{\epsilon,k} \in S(\phi)$.  First, note that
$$
\begin{array}{lll}
\displaystyle \sum_{n\geq 2} V_{n}(\phi^{^{\pm}}_{\epsilon,k}) & = &\displaystyle \sum_{n \geq 2}\sup_{ x,y \in C_{n} }\left\{ \left| \phi^{^{\pm}}_{\epsilon,k}(x) - \phi^{^{\pm}}_{\epsilon,k}(y) \right|\right\}\\\\
&=& \displaystyle \sum_{n \geq 2}\sup_{ x,y \in C_{n} }\left\{ \left| (\phi(x)-\phi(y)) -
\left(B^{^{\pm}}_{\epsilon,k}(x) - B^{^{\pm}}_{\epsilon,k}(y)\right) \right|\right\}\\\\
&=& \displaystyle \sum_{n \geq 2}\sup_{ x,y \in C_{n} }\left\{ \left| \phi(x)-\phi(y)\right|\right\}+\displaystyle \sum_{n \geq 2}\sup_{ x,y \in C_{n} }\left\{ \left| B^{^{\pm}}_{\epsilon,k}(x) - B^{^{\pm}}_{\epsilon,k}(y) \right|\right\}\\\\
&=&\displaystyle \sum_{n\geq 2} V_{n}(\phi) +\sum_{n\geq 0} V_{n}(B^{^{\pm}}_{\epsilon,k}).
\end{array}
$$

%

On the other hand, by definition of $B^{^{\pm}}_{\epsilon,k,l}$, we get 
$$
\begin{array}{lll}
V_{n}(B^{^{\pm}}_{\epsilon,k,l}) 
& = & \displaystyle \sup_{ x,y \in C_{n} }\left\{ \left| B^{^{\pm}}_{\epsilon,k}(x) - B^{^{\pm}}_{\epsilon,k}(y) \right|\right\}\\\\
& \leq &\displaystyle  \sup_{ x,y \in C_{n} }\left\{ \left| B^{^{\pm}}_{\epsilon,k}(x) \right|\right\}+\sup_{ x,y \in C_{n} }\left\{ \left|  B^{^{\pm}}_{\epsilon,k}(y) \right|\right\}\\\\
& \leq & 2\epsilon.
\end{array}
$$

So we have that $\displaystyle \sum_{n\geq 2} V_{n}(\phi^{^{\pm}}_{\epsilon,k})<\infty $, which completes the proof.

\bibliographystyle{ieeetr}
\bibliography{Ref-Gouveia-Oler-Phase}

\begin{thebibliography}{10}

\bibitem{BuzSar2003}
J.~Buzzi and O.~Sarig, ``Uniqueness of equilibrium measures for countable
  {M}arkov shifts and multidimensional piecewise expanding maps,'' {\em Ergodic
  Theory Dynam. Systems}, vol.~23, no.~5, pp.~1383--1400, 2003.

\bibitem{PesZha2006}
Y.~Pesin and K.~Zhang, ``Phase transitions for uniformly expanding maps,'' {\em
  J. Stat. Phys.}, vol.~122, no.~6, pp.~1095--1110, 2006.

\bibitem{BroOle2018}
M.~A. Bronzi and J.~G. Oler, ``Equilibrium state for one-dimensional
  {L}orenz-like expanding maps,'' {\em Bull. Braz. Math. Soc. (N.S.)}, vol.~49,
  no.~4, pp.~873--892, 2018.

\bibitem{G76}
J.~Guckenheimer, ``A strange, strange attractor, in the hopf bifurcation
  theorem and its applications,'' {\em ed. J. Marsden and M. McCracken,
  Springer-Verlag}, pp.~368--381, 1976.

\bibitem{GW79}
J.~Guckenheimer and R.~F. Williams, ``Structural stability of {L}orenz
  attractors,'' {\em Inst. Hautes \'Etudes Sci. Publ. Math.}, no.~50,
  pp.~59--72, 1979.

\bibitem{L63}
E.~N. Lorenz, ``Deterministic non-periodic flow,'' {\em J. Atmos. Sci}, no.~20,
  pp.~130--141, 1963.

\bibitem{S82}
C.~Sparrow, {\em The {L}orenz equations: bifurcations, chaos, and strange
  attractors}, vol.~41 of {\em Applied Mathematical Sciences}.
\newblock New York: Springer-Verlag, 1982.

\bibitem{W79}
R.~F. Williams, ``The structure of {L}orenz attractors,'' {\em Inst. Hautes
  \'Etudes Sci. Publ. Math.}, no.~50, pp.~73--99, 1979.

\bibitem{Gl89}
P.~Glendinning, ``Topological conjugation of {L}orenz maps by
  {$\beta$}-transformations,'' {\em Math. Proc. Cambridge Philos. Soc.},
  vol.~107, no.~2, pp.~401--413, 1990.

\bibitem{GraGrz1998}
J.~Graczyk and G.~Swipolhk~atek, {\em The real {F}atou conjecture}, vol.~144 of
  {\em Annals of Mathematics Studies}.
\newblock Princeton University Press, Princeton, NJ, 1998.

\end{thebibliography}

\medskip
\flushleft
Márcio Gouveia, IBILCE-UNESP, CEP 15054-000, S. J. Rio Preto, São Paulo, Brazil\\
E-mail address:mra.gouveia@unesp.br
\flushleft
\flushleft
Juliano G. Oler, FAMAT-UFU, CEP 38400-902, Uberlândia, Minas Gerais, Brazil\\
E-mail address: jgoler@ufu.br
\flushleft
\vspace{0,1cm}

\end{document}